\documentclass[12pt]{article}

\usepackage{fancyhdr}
\usepackage[lmargin=3cm,rmargin=3cm]{geometry}

\usepackage{graphicx}
\usepackage[brazil]{babel}
\usepackage[utf8]{inputenc}
\usepackage{amsmath,amssymb} 
\usepackage{amsthm} 
\usepackage[table]{xcolor}	
\usepackage{xcolor}
\usepackage{float}
\usepackage[colorlinks=false, bookmarksnumbered=true, bookmarksopen=true, bookmarksopenlevel=3, pdfstartview=FitH, linkcolor=blue, pdfmenubar=true, pdftoolbar=true, bookmarks=true,citecolor=red, urlcolor=blue, filecolor=magenta,plainpages=false,pdfpagelabels,breaklinks]{hyperref}
\usepackage[alf]{abntex2cite}
\usepackage{url}
\usepackage{soul}
\usepackage[tablename=Quadro]{caption}
\newtheorem{teo}{Teorema}[section]

\newtheorem{defi}[teo]{Definição}

\newtheorem{propo}[teo]{Proposi\c c\~ao}

\theoremstyle{definition}

\theoremstyle{remark}

\makeatletter
\renewcommand{\maketitle}{\vspace*{0pt}
	\begin{center}
		 \textbf{\@title}
	\end{center}
	\bgroup\setlength{\parindent}{0pt}
	\begin{flushright}
		\@author
	\end{flushright}\egroup
}

\renewenvironment{abstract}
{{\bfseries\noindent{\normalsize\abstractname}\par\nobreak\smallskip}}

\newenvironment{abstracteng}{
	\begin{flushleft}%
		\bfseries{Abstract}
	\end{flushleft}}%

\numberwithin{equation}{section}

\begin{document}

\setcounter{page}{1}

\title{ {\bf \uppercase{Algumas luminescências sobre o jogo \textit{Lights Out}
}}}

\author{Adriano Verdério, Izabele D'Agostin, Mari Sano, Patrícia Massae Kitani\\{\small Universidade Tecnológica Federal do Paraná}}

\date{}



\maketitle


\begin{abstract}
\noindent A teoria por trás do jogo {\it Lights Out} tem sido desenvolvida por vários autores. O objetivo deste artigo é apresentar alguns resultados relacionados a esse jogo utilizando álgebra linear. Estabelecemos um critério para a solubilidade desse jogo no caso de uma malha $m$ por $n$,
que depende da inversibilidade de uma matriz, e apresentamos as condições para que isso ocorra, de fácil verificação a partir de $m$ e $n$. 
Além disso, determinamos explicitamente o valor do determinante para um caso particular.
\end{abstract}
\vspace{0.2cm}
\noindent 
\textbf{Palavras-chave:}  Matriz tridiagonal em blocos, produto de Kronecker, determinante, jogo  {{\it Lights Out}}.

\begin{abstracteng}
\vspace{-0.2cm}
\noindent The theory behind the Lights Out game has been developed by several authors. The aim of this work is to present some results related to this game using Linear Algebra. We establish a criterion for the solubility of this game in the case of an 
$m$ by $n$ grid,
which depends on the invertibility of a matrix, and we present the conditions for this to occur, easily verifiable from $m$ and $n$.
Furthermore, we explicitly determine the value of the determinant for a particular case.
\end{abstracteng}
\vspace{0.2cm}
\noindent 
\textbf{Keywords:} Block tridiagonal matrix, Kronecker product, determinant, Lights Out game.

\section{Introdu\c{c}\~{a}o}

Em 1995, a {{\it Tiger Electronics}} desenvolveu o jogo eletrônico {{\it Lights Out}}. Ele é basicamente composto por 25 botões posicionados em 5 linhas e 5 colunas como mostra a Figura \ref{fig.1}, alguns iluminados e outros não. O objetivo do jogo é acionar a menor \linebreak quantidade de botões para que todas as luzes sejam desligadas. A {{\it Tiger}} lançou outras versões, como por exemplo o {{\it Mini Lights Out}} que possui o formato de uma malha
$4$ por $4$
e o {{\it Lights Out Deluxe}} com 36 botões, no formato de uma malha 
$6$ por $6$.
Nesse último,  a quantidade de botões que podem ser acionados  é limitada. Anteriormente, em meados de 1970, já havia sido lançado pela {{\it Parker Brothers}} um jogo com regras semelhantes contendo 9 botões, denominado {{\it Merlin}}. 

\begin{figure}[H]
    \centering
    \caption{Jogo \textit{\textbf{Lights Out}}}
    \vspace{-0.4cm}
    \includegraphics[width=0.35\linewidth]{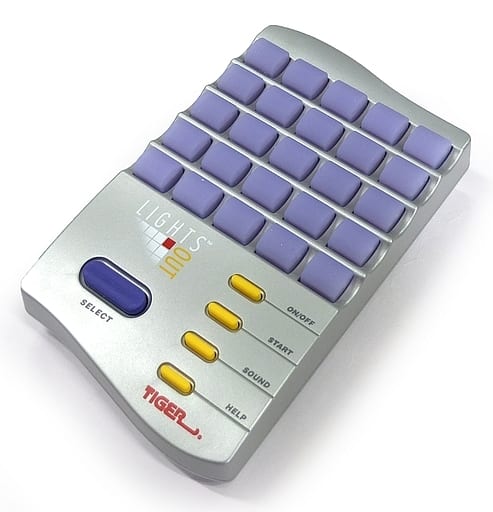}
    \vspace{-0.2cm}
    \centerline{\footnotesize Fonte: {\url{https://www.suruga-ya.com/en/product/148007928}.}}
    \label{fig.1}
\end{figure}

Além desses, existe um jogo que foi adaptado do {\it{Lights Out}} original, criado com o {{\it software}} GeoGebra, denominado: ``Acenda e apague a luz", que consiste em reproduzir padrões, onde todas ou determinadas luzes devem ficar acesas, dependendo do objetivo do jogo. O jogo está disponível no {{\it link}}:

\begin{center}
\url{http://clubes.obmep.org.br/blog/acenda-e-apague-a-luz/comment-page-1/}.
\end{center}\medskip

Entre os trabalhos relacionados com o jogo {\it{Lights Out}} e algumas de suas várias versões
destacamos os desenvolvidos por \hyperlink{sutner}{Klaus Sutner (1989)}; \hyperlink{barua}{Rana Barua e Subramanian Ramakrishnan (1996)}; \hyperlink{anderson}{Marlow Anderson e Todd Feil (1998)}; 
\hyperlink{goshima}{Masato Goshima e Masakazu Yamagishi (2009)}; \hyperlink{madsen}{Matthew A. Madsen (2010)}; \hyperlink{fleischer}{Rudolf Fleischer e Jiajin Yu (2013)};
\hyperlink{kreh}{Martin Kreh (2017)} e \hyperlink{berman}{Abraham Berman, Franziska Borer e Norbert Hungerbühler (2021)}.

\hyperlink{sutner}{Klaus Sutner (1989)} resolve o jogo para a malha
$n$ por $n$ 
usando modelos matemáticos autômatos celulares e teoria de grafos.

\hyperlink{barua}{Rana Barua e Subramanian Ramakrishnan (1996)} fornecem uma condição necessária e suficiente para que o jogo tenha solução, no caso da malha 
$m$ por $n$,
relacionando a solubilidade com o fato que certos polinômios de Fibonacci (envolvendo $n$ e $m$) sejam coprimos. Também fornecem um algoritmo para encontrar o número de soluções possíveis para uma dada configuração inicial.

\hyperlink{anderson}{Marlow Anderson e Todd Feil (1998)} estudam a solubilidade do jogo no contexto da álgebra linear principalmente para o caso da malha 
$5$ por $5$,
relacionando a configuração inicial do jogo com a base do espaço nulo de uma certa matriz. Além disso, eles fazem a observação de que esse estudo pode ser generalizado para o caso da malha $n\times n$ usando métodos análogos.

\hyperlink{goshima}{Masato Goshima e Masakazu Yamagishi (2009)} demonstram que dada uma configuração inicial do jogo {{\it Torus Lights Out}} na malha 
$2^k$ por $2^k$, 
para $k\geq 3$, este pode ser resolvido repetindo um determinado procedimento utilizando a teoria de álgebra linear e a teoria de grafos. Ainda exibem um critério de solubilidade para o jogo {{\it Torus Lights Out}} na malha 
$5^k$ por $5^k$, 
com $k\geq 2$.

\hyperlink{madsen}{Matthew A. Madsen (2010)} aborda vários métodos diferentes que podem ser usados para encontrar uma solução para o jogo no caso da malha
$5$ por $5$
e utiliza álgebra linear para encontrar uma estratégia vencedora usando o menor número de movimentos.

\hyperlink{fleischer}{Rudolf Fleischer e Jiajin Yu (2013)} fazem uma revisão dos trabalhos referente ao jogo em uma apresentação unificada.

\hyperlink{kreh}{Martin Kreh (2017)} estabelece critérios de solubilidade para o jogo com malha
$n$ por $n$ 
e além das luzes estarem somente acesas ou apagadas, são consideradas luzes coloridas que 
vão mudando de forma cíclica. Para a demonstração desses critérios foram usadas técnicas 
da álgebra linear. Ademais foram discutidas maneiras de estudar a solubilidade do jogo 
baseadas na teoria algébrica dos números.

\hyperlink{berman}{Abraham Berman, Franziska Borer e Norbert Hungerbühler (2021)} analizam novas versões do jogo e também consideram luzes coloridas. A solubilidade do jogo recai em um sistemas de equações lineares cuja solução envolve a decomposição de Smith.

Neste trabalho, analisamos o resultado para o caso da malha
$m$ por $n$,
considerando apenas luzes apagadas ou acesas. A solução depende se uma certa matriz tridiagonal em blocos é invertível ou não. Para o caso da malha
$2$ por $n$, 
determinamos ainda o valor do determinante. Para demonstrar o primeiro resultado, utilizamos ferramentas da álgebra linear e, para o segundo, fizemos uso de algumas identidades trigonométricas básicas.

\section{Preliminares}

\subsection{A álgebra linear do jogo \textit{\textbf{Lights Out}}}
\label{al_lo}

Nesta seção, explicamos brevemente como o jogo funciona no formato de uma malha
$m$ por $n$ 
com duas cores para a luz (acesa/apagada) e como a álgebra linear está envolvida na resolução do jogo {{\it Lights Out}}. 

Nesse caso, podemos representar os botões do jogo pelos elementos de uma matriz. Cada botão traz uma luz, que pode estar acesa ou apagada. Quando um botão é pressionado, seu estado  ligado/desligado é alterado, assim como, de todos os botões vizinhos verticais e horizontais. 



Então, dada uma configuração inicial de botões acesos, o objetivo do jogo é apagar todas as luzes.
Vamos admitir que pressionar um botão duas vezes é equivalente a não pressioná-lo. Dessa forma, dada uma configuração, vamos considerar apenas as soluções nas quais cada botão é pressionado uma única vez. Além disso, o estado (ligado/desligado) de um botão depende de quantas vezes ele e seus vizinhos foram acionados. A Figura \ref{fig.2} mostra um exemplo do jogo na malha 
$3$ por $3$.

\begin{figure}[H]
    \centering
    \caption{Jogo na malha $\mathbf{3\times 3}$}
    \vspace{-0.4cm}
    \includegraphics[width=0.9\linewidth]{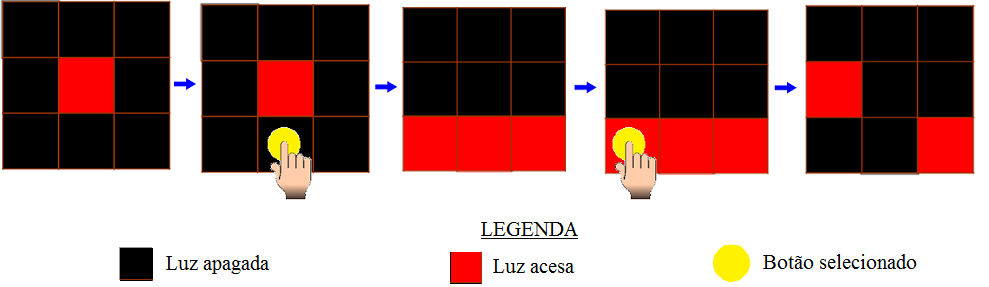}
    \centerline{\footnotesize Fonte: Os autores}
    \label{fig.2}
\end{figure}

\vspace{-0.3cm}

A seguir, disponibilizamos o link: 

\begin{center}
\url{https://www.geogebra.org/m/rbwagfqj}
\end{center}
que contém as versões do jogo {{\it Lights Out}} nas malhas: 
$2$ por $3$, $3$ por $3$ e $4$ por $4$,
elaboradas com o auxílio do {{\it software}} GeoGebra. Ao acessar o {{\it link}} e escolher o formato do jogo, é necessário clicar no botão sortear para iniciar o jogo. Após clicar no botão sortear, será fornecida uma configuração inicial
de luzes acesas (vermelhas) e apagadas (pretas). Depois é só escolher os botões, afim de que
todas as luzes sejam apagadas.

 Como há exatamente dois estados, ligado ou desligado, vamos representá-los por elementos de $\mathbb{Z}_2$ (anel dos números inteiros módulo $2$). Assim, quando a luz estiver na posição ligada será representada pelo número $1$, caso contrário será representada pelo número $0$. 
 
 Posto isto, vamos estabelecer uma estratégia para resolver o jogo com $q=mn$ botões disponibilizados no formato retangular, utilizando tópicos de álgebra linear. Suponha que o jogo tenha $mn$ botões numeradas por linhas, como no Quadro \ref{Tab.1} a seguir.

\begin{table}[h!]

\begin{center} 
\caption{Enumeração das teclas}
\vspace{-0.3cm}
\setlength{\arrayrulewidth}{0.35 mm}
\label{Tab.1}
\begin{tabular}{||c||c||c||c|| c||}\hline
1 & 2& $\cdots$ &$n-1$ & $n$ \\ \hline \hline
$n+1$ & $n+2$ & $\cdots$ & $2n-1$ & $2n$ \\  \hline \hline
$2n+1$ & $2n+2$ & $\cdots$ & $3n-1$ &  $3n$ \\  \hline \hline
$\vdots$ & $\vdots$ & $\vdots$ & $\vdots$ &$\vdots$ \\  \hline \hline
$q-n+1$ & $q-n+2$ & $\cdots$ & $q-1$ & $q$ \\  \hline 
\end{tabular} \vspace{-0.4cm}
\centerline{\footnotesize \bf Fonte: Os autores}
\end{center}
\end{table}

Chamamos de $C\in (\mathbb{Z}_2)^q$ o vetor que representa a configuração inicial (de luzes acesas ou apagadas). A $i$-\'esima coordenada deste vetor indica se a luz da posição $i$ est\'a acesa ou não.  Por exemplo, se inicialmente todas as luzes estão acesas, então $C=(1,1,1, \dots, 1,1),$ ou então, se apenas as luzes $2$ e $q-1$ est\~ao acesas teremos $C=(0,1,0,0, \dots, 0, 1, 0).$

Denotaremos por $A_j\in (\mathbb{Z}_2)^q$, com $j=1,2,\dots,q$, o vetor que aponta quais luzes são acesas quando o botão $j$ é acionado, considerando que todos os botões estão apagados inicialmente, e 
$x_i = 1$ ou $x_i = 0$, 
com $i=1,2,\dots,q$, que indica se o botão $i$ foi acionado ou não,
respectivamente.
Por exemplo, se o jogo tem o formato da malha
$m$ por $n$ 
e se o botão $1$ for acionado, então $A_1=(1,1, 0, \dots, 0, 1, 0, 0, 0, \dots, 0, 0),$ com $1$ nas posições $1, 2$ e $n+1$ como mostra o Quadro \ref{Tab.2}.

\begin{table}[h!]
\begin{center} 
\caption{Botão $\mathbf{1}$ acionado}
\vspace{-0.3cm}
\setlength{\arrayrulewidth}{0.35mm}
\label{Tab.2}
\begin{tabular}{||c||c||c||c|| c||}\hline
\cellcolor{yellow!} $\mathbf{1}$ &\cellcolor{yellow!} $\mathbf{2}$ & $\cdots$ & $n-1$ & $n$ \\ \hline \hline
\cellcolor{yellow!}$\mathbf{n+1}$ & $n+2$ & $\cdots$ & $2n-1$ & $2n$ \\  \hline \hline
$2n+1$ & $2n+2$ & $\cdots$ & $3n-1$ &  $3n$ \\  \hline \hline
$\vdots$ & $\vdots$ & $\vdots$ & $\vdots$ & $\vdots$ \\  \hline \hline
$q-n+1$ & $q-n+2$ & $\cdots$ & $q-1$ & $q$ \\  \hline 
\end{tabular} \vspace{-0.4cm}
\centerline{\footnotesize \bf Fonte: Os autores}
\end{center}
\end{table}

De modo geral, $A_j$ possui $1$ nas posições $j-n, j-1,j, j+1$ e $j+n,$ e $0$ nas demais posições, observando que quando os botões acionados estão nas laterais (ou linha $1,$ ou linha $m,$ ou coluna $1$ ou coluna $n$), somente 3 ou 4 entradas terão o valor $1$. Portanto, teremos sempre no mínimo $3$ e no máximo $5$ botões que são alterados.

Assim, $x_jA_j$ com $j=1,2,\dots,q$, indica que o botão $j$ foi selecionado, quando $x_j=1$ e, portanto, temos a configuração de quais botões foram acesos e apagados nesse processo, ou apenas o vetor nulo se $x_j=0$, que indica que o botão não foi selecionado.
Dessa forma, precisamos determinar os valores de $x_j$, com $j=1,2,\dots,q$, de modo que ao adicionarmos a configuração inicial com cada $x_jA_j$, onde $j=1,2,\dots, q$,
 o resultado seja o vetor nulo. Em outras palavras, devemos resolver a seguinte equação:
$$C+x_1A_1+x_2A_2+\dots+ x_qA_q= \Vec{0}.$$ Como $C\equiv -C \ (mod~2)$, podemos escrever $$x_1A_1+x_2A_2+\dots+ x_qA_q= C.$$ 

Observe que a ordem em que os botões são acionados não faz diferença, uma vez que a adição em $\mathbb{Z}_2$ é comutativa como mostra a Figura \ref{fig.3}:

\begin{figure}[H]
    \centering
    \caption{Comutatividade do jogo na malha $\mathbf{3}$ por $\mathbf{3}$}
    \vspace{-0.4cm}
    \includegraphics[width=1\linewidth]{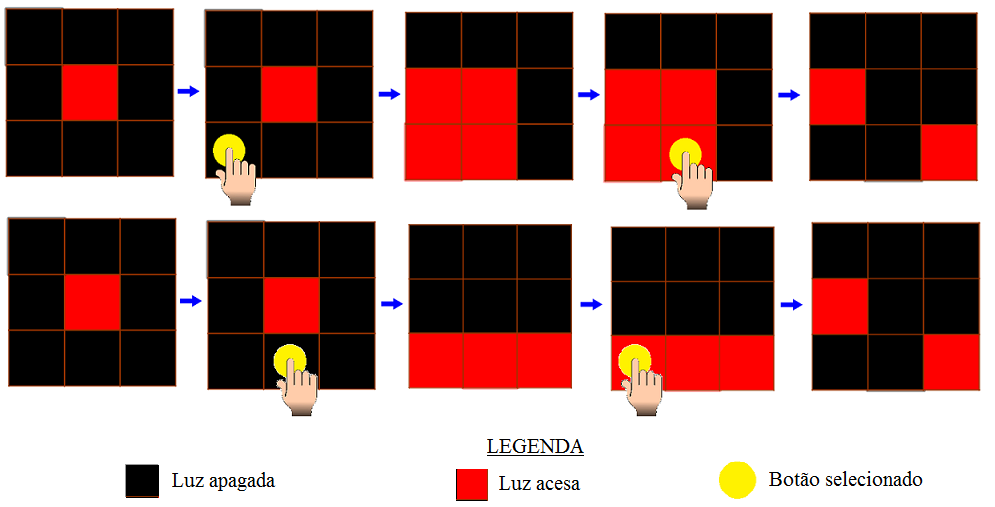}
    \centerline{\footnotesize \bf Fonte: Os autores}
    \label{fig.3}
\end{figure}

\vspace{-0.4cm}

Note que, se chamamos de $A(m,n)$ a matriz formada pelas colunas $A_1, A_2,\dots, A_q$ e $X$ a matriz coluna formada por $x_1, x_2, \dots, x_q$, podemos representar a equação acima da seguinte forma: $A(m,n) \cdot X=C$.

A matriz $A(m,n)$ será sempre uma matriz tridiagonal em blocos de ordem $m n$, onde os blocos são de ordem $n$ e podem ser de 3 formas: banda $1$ ou tridiagonal (matriz quadrada onde apenas os elementos da diagonal principal e das diagoniais que estão logo acima e abaixo dela são não nulos) $T_n$ com $1$ nas entradas das diagonais, identidade $I_n$ e nula $0_n$. 
Ou seja,

\[ A(m,n) = \left[ \begin{array}{ccccccccc}
T_n & I_n & 0_n & 0_n & \cdots & 0_n & 0_n & 0_n  & 0_n\\
I_n & T_n & I_n & 0_n & \cdots & 0_n & 0_n & 0_n  & 0_n \\
0_n & I_n & T_n & I_n & \cdots & 0_n & 0_n & 0_n & 0_n\\
0_n & 0_n & I_n & T_n & \cdots & 0_n & 0_n & 0_n & 0_n\\
\vdots & \vdots & \vdots & \vdots & \ddots & \vdots & \vdots & \vdots & \vdots \\
0_n & 0_n & 0_n & 0_n & \cdots & T_n & I_n & 0_n & 0_n \\
0_n & 0_n & 0_n & 0_n & \cdots & I_n & T_n & I_n & 0_n \\
0_n & 0_n & 0_n & 0_n & \cdots & 0_n & I_n & T_n & I_n \\
0_n & 0_n & 0_n &  0_n & \cdots & 0_n & 0_n &I_n & T_n 
\end{array} \right]. \]

Por exemplo, se $m = 5$ e $n = 2$, a matriz será de ordem $5\cdot 2=10$ da seguinte forma:

\[ A(5,2) = \left[ \begin{array}{ccccc}
T_2 & I_2 & 0_2 & 0_2 & 0_2 \\
I_2 & T_2 & I_2 & 0_2 & 0_2 \\
0_2 & I_2 & T_2 & I_2 & 0_2 \\
0_2 & 0_2 & I_2 & T_2 & I_2 \\
0_2 & 0_2 & 0_2 & I_2 & T_2 
\end{array} \right] = \left[
\begin{tabular}{c c | c c  |c c | c c  |c c  }
1 & 1 & 1 & 0 & 0 & 0 & 0 & 0 & 0 & 0 \\
1 & 1 & 0 & 1 & 0 & 0 & 0 & 0 & 0 & 0\\ \hline
1 & 0 & 1 & 1 & 1 & 0 & 0 & 0 & 0 & 0\\
0 & 1 & 1 & 1 & 0 & 1 & 0 & 0 & 0 & 0 \\ \hline
0 & 0 & 1 & 0 & 1 & 1 & 1 & 0 & 0 & 0 \\
0 & 0 & 0 & 1 & 1 & 1 & 0 & 1 & 0 & 0 \\\hline
0 & 0 & 0 & 0 & 1 & 0 & 1 & 1 & 1 & 0 \\
0 & 0 & 0 & 0 & 0 & 1 & 1 & 1 & 0 & 1\\ \hline 
0 & 0 & 0 & 0 & 0 & 0 & 1 & 0 & 1 & 1\\
0 & 0 & 0 & 0 & 0 & 0 & 0 & 1 & 1 & 1 
\end{tabular}
\right] \]

e se $m = 2$ e $n = 5$, a matriz será de ordem $2\cdot 5=10$ da seguinte forma:
\[ A(2,5) = \left[ \begin{array}{cc}
T_5 & I_5 \\
I_5 & T_5 
\end{array} \right]=  \left[
\begin{tabular}{c c c c c | c c c c c  }
1 & 1 & 0 & 0 & 0 & 1 & 0 & 0 & 0 & 0 \\
1 & 1 & 1 & 0 & 0 & 0 & 1 & 0 & 0 & 0\\ 
0 & 1 & 1 & 1 & 0 & 0 & 0 & 1 & 0 & 0\\
0 & 0 & 1 & 1 & 1 & 0 & 0 & 0 & 1 & 0 \\ 
0 & 0 & 0 & 1 & 1 & 0 & 0 & 0 & 0 & 1 \\ \hline
1 & 0 & 0 & 0 & 0 & 1 & 1 & 0 & 0 & 0 \\
0 & 1 & 0 & 0 & 0 & 1 & 1 & 1 & 0 & 0 \\
0 & 0 & 1 & 0 & 0 & 0 & 1 & 1 & 1 & 0\\ 
0 & 0 & 0 & 1 & 0 & 0 & 0 & 1 & 1 & 1\\
0 & 0 & 0 & 0 & 1 & 0 & 0 & 0 & 1 & 1 
\end{tabular}
\right].\]

Ou seja, a matriz das configurações do jogo de uma malha de $m$ linhas e $n$ colunas, será a matriz quadrada $A(m,n)$ de ordem $m\cdot n$ que terá $m^2$ blocos, cada um formado por matrizes de ordem $n$.







 A partir da conexão estabelecida entre o jogo {{\it Lights Out}} e a teoria da álgebra linear, podemos questionar: sob quais condições será possível apagar todas as luzes? 

\hyperlink{anderson}{Marlow Anderson e Todd Feil (1998)} mostram que o jogo tem solução se o vetor que representa a configuração inicial é ortogonal ao espaço nulo da matriz $A(n,n)$. Tal resultado se mantém para o caso de $A(m,n)$, uma vez que o desenvolvimento é exatamente o mesmo. Vale ressaltar que \hyperlink{anderson}{Marlow Anderson e Todd Feil (1998, p. 303)} apresentam uma tabela com a dimensão do espaço nulo da matriz $A(n,n)$ com um erro no caso $n = 16$, o valor correto para a dimensão do espaço nulo é zero, uma vez que o determinante não é nulo conforme calculado por \hyperlink{kreh}{Martin Kreh (2017, p. 942)}.

{ Por exemplo, para o jogo com o formato de uma malha 
$2$ por $5$,
temos que o determinante da matriz $A(2,5)$ é igual a zero, ou seja, ou o sistema $A(2,5) \cdot X = C$ não possui solução ou possui infinitas soluções. Considerando a configuração inicial $C_0 =(0,1,0,1,0,0,1,0,1,0)$, em que as segunda e quarta colunas de botões estão acesas, obtemos que o jogo possui duas soluções: apertar uma vez os botões $1$, $5$, $6$ e $10$ ou apertar uma vez os botões $3$ e $8$ (onde os botões estão numerados conforme Quadro \ref{Tab.1}). Podemos mostrar que o sistema $A(2,5) \cdot X = C_0$ admite infinitas soluções, que são aquelas representadas por apertar uma quantidade ímpar de vezes os botões em cada uma das soluções apresentadas e uma quantidade par de vezes os demais botões. Por outro lado, se a configuração inicial for $C'=(1,0,0,0,0,0,0,0,0,0)$ o sistema $A(2,5) \cdot X = C'$ não admite solução.}

 Nosso objetivo aqui é estabelecer quais as condições para que o jogo apresente solução independentemente da configuração inicial. Sabemos que $A(m,n) \cdot X=C$ tem solução quando a matriz $A(m,n)$ é invertível, portanto, precisamos verificar quando isto ocorre.


\subsection{O determinante de $\mathbf{A(m,n)}$}

Primeiramente, nesta seção, definimos o produto e a soma de Kronecker para elucidar a fórmula do determinante da matriz $A(m,n)$, a partir de seus autovalores. Começamos pela definição de produto de Kronecker que pode ser  encontrada no livro de \hyperlink{horn}{Roger A. Horn e Charles R. Johnson (1991, p. 243)}.

\begin{defi} Sejam $A = [a_{ij}]$ uma matriz $m \times n$ e $B = [b_{ik}]$ uma matriz $p \times q$  definimos o produto de Kronecker de $A$ e $B$ (nessa ordem) como a matriz $mp \times nq$ em blocos dada por
\[
A\otimes B=  \left[ \begin{array}{cccc}
a_{11} B & a_{12} B & \dots &a_{1n} B \\
a_{21} B & a_{22} B & \dots &a_{2n} B \\
\vdots & \vdots & \ddots & \vdots \\
a_{m1} B & a_{m2} B & \dots &a_{mn} B
\end{array} \right].
\]
\end{defi}




A partir da definição anterior é fácil ver que
\[
I_m \otimes T_n =  \left[ \begin{array}{cccc}
T_n & 0 & \dots & 0 \\
0 & T_n & \dots & 0 \\
\vdots & \vdots & \ddots & \vdots \\
0 & 0 & \dots & T_n
\end{array} \right],
\]
de ordem $m n$; 
\[
T_m \otimes I_n =  \left[ \begin{array}{ccccc}
I_n & I_n & 0 & \dots & 0 \\
I_n & I_n & I_n & \dots & 0 \\
\vdots & \ddots & \ddots & \ddots & \vdots \\
0 & \dots & I_n & I_n & I_n \\
0 & 0 & \dots & I_n & I_n
\end{array} \right],
\]
também de ordem $mn$, e que $I_m \otimes I_n = I_{mn}$.

Assim, temos que
\[ A(m,n) = I_m \otimes T_n + T_m \otimes I_n - I_{mn}. \]

Já a soma de Kronecker \hyperlink{horn}{(Horn e Johnson, 1991, p. 268)} está definida para matrizes quadradas da seguinte forma: se $A$ tem ordem $n$ e $B$ tem ordem $m$, então 
\[ A \oplus B = I_m \otimes A + B \otimes I_n \]
é uma matriz de ordem $mn$. Com isso, temos que
\[ A(m,n) = T_n \oplus T_m - I_{mn}. \]

Agora, dada uma matriz $M$, seus autovalores satisfazem a equação característica $det(M - xI) = 0$ \hyperlink{leon}{(Leon, 2014, p. 266)}.  Mas como, 
\[ 0 = det(M - xI) = det(M - ((x-1)I + I) = det((M - I) - (x-1)I) ,\]
temos que os autovalores de $M - I$ são os autovalores de $M$ menos 1.

Ou seja, para encontrar os autovalores de $A(m,n)$ precisamos determinar os autovalores de $T_n \oplus T_m$. Mas, segundo \hyperlink{horn}{Roger A. Horn e Charles R. Johnson (1991, p. 269)}, os autovalores da soma de Kronecker são a soma dos autovalores das matrizes envolvidas.



Desta maneira, precisamos dos autovalores de $T_n$, que de acordo com \hyperlink{noschese}{Silvia Noschese, Lionello Pasquini e Lothar Reichel (2013, p. 304)}, são da forma 
\[ \lambda_k = 1 + 2 \cos \left( \frac{k\pi}{n+1} \right) \quad \text{ para } k = 1, 2, \dots, n.\]




Portanto, temos que os autovalores de $A(m,n)$ são da forma
\begin{equation} \label{autovalores}
    \lambda_{jk} = 1 + 2 \left[ \cos\left(\frac{j\pi}{m+1}\right) + \cos\left(\frac{k\pi}{n+1}\right) \right] 
\end{equation} 
com $j = 1, 2, \dots, m$ e $k = 1, 2, \dots, n$.

\ 

\noindent Lembrando que o determinante de uma matriz é o produto de seus autovalores \hyperlink{leon}{(Leon, 2014, p. 271)} temos que

\[ det(A(m,n)) = \prod_{j=1}^m \prod_{k=1}^n \lambda_{jk}.\]

\ 

\noindent Logo,
\[ det(A(m,n)) = \prod_{j=1}^m \prod_{k=1}^n \left[ 1 + 2 \left( \cos\left(\frac{j\pi}{m+1}\right) + \cos\left(\frac{k\pi}{n+1}\right) \right) \right].\]

Usando a fórmula obtida acima podemos observar que o determinante das matrizes $A(m,n)$ e $A(n,m)$ são iguais, ou seja, rotacionar o jogo não interfere em sua solubilidade, como era de se esperar.


\section{Alguns resultados sobre o jogo \textit{\textbf{Lights Out}}}

Nesta seção, abordamos os principais resultados deste trabalho e o teorema que empregamos para demonstrar um deles.

Seja $r\in \mathbb{Q}.$ Denota-se por $N(r)$ o menor inteiro positivo tal que $rN(r)$ seja um inteiro, ou seja, $N(r)=q$ se $r=\dfrac{p}{q}$ com $p$ e $q> 0$ inteiros primos entre si. Assim enunciamos o seguinte teorema, que pode ser encontrado no artigo de \hyperlink{berger}{Arno Berger (2018)}.

\begin{teo} \label{Teo1}
Sejam $r_1, r_2\in \mathbb{Q}$ tais que $r_1-r_2\not\in \mathbb{Z}$ e $r_1+r_2\not\in \mathbb{Z}.$ Então são equivalentes:
\begin{enumerate}
    \item[i)] Os números $1, \ \cos(r_1 \pi), \ \cos(r_2\pi)$ são linearmente independentes sobre $\mathbb{Q};$
    \item[ii)] $N(r_j)\geq 4$ para $j\in\{1,2\}$ e $(N(r_1),N(r_2))\neq (5,5).$
\end{enumerate}
\end{teo}

A seguir apresentamos o resultado central deste trabalho, o qual fornece um critério para a solubilidade do jogo {{\it Lights Out}}, independente da configuração inicial.

\begin{teo} \label{Teo2}
A matriz $A(m,n)$ é singular se, e somente se, vale uma das seguintes afirmações:
\begin{enumerate}
\item $m \equiv 2 \  (mod~3)$ e $n$ é ímpar.
\item $m$ é ímpar e $n \equiv 2 \  (mod~3)$.
\item $m \equiv 4 \  (mod~5)$ e $n \equiv 4 \ (mod~5)$.
\end{enumerate}
\end{teo}

\ 

\begin{proof}

 Suponha que $A(m,n)$ seja singular. Sabemos que $A(m,n)$ é singular se, e somente se, um de seus autovalores é zero. Usando \eqref{autovalores}, temos que
\[ \lambda_{jk} = 1 + 2 \left[ \cos\left(\dfrac{j\pi}{m+1}\right) + \cos\left(\dfrac{k\pi}{n+1}\right) \right] = 0,\]
para algum $j=1,2, \dots, m$ e algum $ k=1, 2, \ldots, n.$ Isto é, devemos ter que

\begin{equation} \label{condição}
\left[ \cos\left(\frac{j\pi}{m+1}\right) + \cos\left(\frac{k\pi}{n+1}\right) \right]=-\dfrac{1}{2},
\end{equation}
para algum $j=1,2, \dots, m$ e algum $ k=1, 2, \dots, n.$

Sejam $\alpha=\dfrac{j}{m+1}$ e $\beta=\dfrac{k}{n+1}$. Vamos analisar dois casos: $\alpha+\beta, \alpha-\beta \in \mathbb{Z}$ e $\alpha+\beta, \alpha-\beta \notin \mathbb{Z}$.

\begin{enumerate}
    \item[Caso 1.] Se $\alpha+\beta, \alpha-\beta \in \mathbb{Z}$, então usando a identidade 
\[
\cos (\theta+\phi)+\cos (\theta-\phi)=2\cos (\theta)\cos(\phi)
\]
para $\theta=(\alpha+\beta)\frac{\pi}{2}$ e $\phi=(\alpha-\beta)\frac{\pi}{2}$ temos que 
\[
 \cos(\alpha\pi)+ \cos(\beta \pi)=2 \cos\left(\left(\alpha+\beta\right)\dfrac{\pi}{2}\right)\cos\left(\left(\alpha-\beta\right)\dfrac{\pi}{2}\right)
\]
 que por sua vez
é igual a $0$, $2$ ou $-2$, que são diferentes de $-\dfrac{1}{2}.$
\item[Caso 2.] Supondo que $\alpha+\beta, \alpha-\beta \notin \mathbb{Z}$ e considerando que uma condição necessária para que a igualdade \eqref{condição} seja verdadeira é que $1, \cos(\alpha\pi)$ e $ \cos(\beta \pi)$ sejam linearmente dependentes sobre $\mathbb{Q}$. 
Pelo Teorema \ref{Teo1} podemos afirmar que 
$N(\alpha)<4$ e $N(\beta)<4$ ou $N(\alpha)=N(\beta)=5.$ Consequentemente $\alpha$  e $\beta$  pertencem ao conjunto

 \[	\left\{ 0, \frac{1}{5}, \frac{1}{3}, \frac{2}{5}, \frac{1}{2}, \frac{3}{5}, \frac{2}{3}, \frac{4}{5}, 1 \right\}.\]

Em particular, queremos
\[ \cos(\alpha \pi) + \cos(\beta \pi)=-\dfrac{1}{2}.\] 

Como $0< \alpha, \beta < 1,$ podemos desconsiderar $\alpha$ e $\beta$ iguais a $0$ e $1.$
Assim, vamos analisar o conjunto
\[ X=	\left\{ \frac{1}{5}, \frac{1}{3}, \frac{2}{5}, \frac{1}{2}, \frac{3}{5}, \frac{2}{3}, \frac{4}{5} \right\},\]
de modo que $\alpha$ e $ \beta$ satisfaçam $\alpha\pm \beta \not\in\mathbb{Z}.$
Considerando todas as possíveis combinações de $\alpha$ e $\beta$ no conjunto $X$ vemos que as únicas possibilidades que procuramos são as seguintes:

\begin{table}[ht]
\renewcommand*{\arraystretch}{1.8}
	\centering
	\begin{tabular}{|l|l|c|c|c|c|}
	 \hline 
		$\alpha$        & $\beta$        & $\cos(\alpha \pi)$ & $\cos(\beta \pi)$ & $\cos(\alpha \pi)+\cos(\beta \pi)$   \\
		\hline 
		$\frac{1}{2}$ & $\frac{2}{3}$ & $0$ &  $-\frac{1}{2}$ & $-\frac{1}{2}$ \\
		\hline 
		$\frac{2}{3}$ & $\frac{1}{2}$ & $-\frac{1}{2}$ &  $0$ & $-\frac{1}{2}$ \\
		\hline 
		$\frac{2}{5}$ & $\frac{4}{5}$ & $\frac{-1+\sqrt{5}}{4}$ &  $\frac{-1-\sqrt{5}}{4}$ & $-\frac{1}{2}$ \\
		\hline 
		$\frac{4}{5}$ & $\frac{2}{5}$ & $\frac{-1-\sqrt{5}}{4}$ & $\frac{-1+\sqrt{5}}{4}$ & $-\frac{1}{2}$ \\
		\hline
	\end{tabular}
\end{table}

 A primeira linha da tabela nos dá a condição
\[ m \,\,\,\,\, \textrm{ímpar} \,\,\,\,\, \textrm{e} \,\,\,\,\, n \equiv 2 \ (mod~3),\]
 uma vez que para que 
$ \alpha = \frac{j}{m+1} = \frac{1}{2} $
precisamos $m = 2j - 1$ e para 
$ \beta = \frac{k}{n+1} = \frac{2}{3} $
precisamos de $n =  \frac{3k}{2} - 1$ e $k$ par (ou seja, $n = 3 \left(\frac{k}{2}-1\right) + 2$). Analogamente,
da segunda linha obtemos a condição
\[
n \,\,\,\,\, \textrm{ímpar}  \,\,\,\,\, \textrm{e} \,\,\,\,\, m \equiv 2 \ (mod~3)\]
e
com o mesmo raciocínio,
nas terceira e quarta linhas, segue que
\[ m \equiv 4 \ (mod~5) \,\,\,\,\,  \textrm{e} \,\,\,\,\, n \equiv 4\ (mod~5).\] 
\end{enumerate}

Na outra direção, supondo que $m \equiv 2 \  (mod~3)$ e $n$ é ímpar temos que 
$m=2+3q$ e $n=1+2p$ para alguns $q,p\in \mathbb{N}$. Portanto, para $1 \leq j = 2(q+1) \leq m$ e para $1 \leq k = p+1 \leq n$,

\begin{eqnarray*}
\cos\left(\dfrac{2(q+1)\pi}{2 + 3q + 1}\right) + \cos\left(\dfrac{(p+1)\pi}{1+2p+1}\right) 
& = & \cos\left(\dfrac{2(q+1)\pi}{3(q + 1)}\right) + \cos\left(\dfrac{(p+1)\pi}{2(p+1)}\right) \\
& = & \cos \left(\frac{2\pi}{3}\right) + \cos\left(\frac{\pi}{2}\right)=-\dfrac{1}{2}.
\end{eqnarray*}

O caso $m$ é ímpar e $n \equiv 2 \  (mod~3)$ é análogo ao anterior.

Agora, suponhamos que $m \equiv 4 \ (mod~5)$ e $n \equiv 4 \ (mod~5)$. Logo, $m = 4+5q$ e $n = 4+5p$ para alguns $q, p\in \mathbb{N}$. Portanto, para $1 \leq j = 2(q+1) \leq m$ e para $1 \leq k = 4(p+1) \leq n$,

\begin{eqnarray*}
\cos\left(\dfrac{2(q+1)\pi}{4 + 5q + 1}\right) + \cos\left(\dfrac{4(p+1)\pi}{4+5p+1}\right) 
& = & \cos \left(\frac{2\pi}{5}\right) + \cos\left(\frac{4\pi}{5}\right)=-\dfrac{1}{2}.
\end{eqnarray*}
concluindo o resultado.

\end{proof}

Podemos observar que, se nenhuma das três afirmações do Teorema \ref{Teo2} é satisfeita segue que o jogo tem solução independente da configuração inicial. Mas, se a matriz $A(m,n)$ é singular não podemos afirmar que o jogo não tem solução (depende da configuração inicial) como já apresentado anteriormente na Subseção \ref{al_lo}. 


Por fim, mostraremos uma fórmula mais simples para o determinante da matriz $A(2,n).$ O problema de encontrar determinantes de certas matrizes é um problema clássico e aqui, encontramos o determinante de uma matriz tridiagonal em blocos, que por sua vez, é também um caso particular de uma matriz pentadiagonal. Para tanto, precisaremos de uma identidade trigonométrica que será demonstrada na próxima proposição.

\begin{propo} \label{prop1}
Se $n\in \mathbb{N}$, então
    \[
    \displaystyle \prod_{k=1}^{n-1} \cos\left(\dfrac{\pi k}{n}\right)=\dfrac{\mathrm{sen}\left(\dfrac{\pi n}{2}\right)}{2^{n-1}}.
    \]
\end{propo} 
    \begin{proof}
        As raízes da equação $z^{2n} -1=0$ são dadas por $e^{\frac{i2\pi k}{2n}},$ onde \linebreak $k=0,1,2,\dots,2n-1$. Escrevendo $w=e^{\frac{i\pi}{n}}$, temos que

        {\footnotesize}
        \begin{align*}
           z^{2n} -1&= \displaystyle \prod_{k=0}^{2n-1} (z-w^k)\\
           &=(z-1)(z-w)\cdots(z-w^{n-1})(z-w^n)(z-w^{n+1})\cdots (z-w^{2n-1}).
           \end{align*}
           
         Como $w^{2n-k}=w^{-k}$, temos que

         \begin{align*}
          z^{2n} -1&=(z-1)(z+1)\displaystyle \prod_{k=1}^{n-1}(z-w^k)(z-w^{-k})\\
           &=(z-1)(z+1)\displaystyle \prod_{k=1}^{n-1} \left(z^2 -2\cos\left(\dfrac{\pi k}{n}\right) z+1\right).
        \end{align*}
        {\footnotesize}

        Para $z=i$, segue que

        \begin{align*}
            (-1)^n -1=i^{2n}-1&=(-2)\displaystyle \prod_{k=1}^{n-1} (-2) \cos\left(\dfrac{\pi k}{n}\right) i
            =(-2)(-2i)^{n-1}\displaystyle \prod_{k=1}^{n-1} \cos\left(\dfrac{\pi k}{n}\right).
        \end{align*}
        De onde, $\displaystyle \prod_{k=1}^{n-1} \cos\left(\dfrac{\pi k}{n}\right)=\dfrac{(-1)^n -1}{(-2)(-2i)^{n-1}}$ é igual a $0$ se $n$ é par. Se $n$ é ímpar ($n=2r+1$, para algum $r\in \mathbb{N}$), obtemos

        \[
        \displaystyle \prod_{k=1}^{n-1} \cos\left(\dfrac{\pi k}{n}\right)=\left(\dfrac{1}{-2i}\right)^{n-1}=\left(\dfrac{i}{2}\right)^{n-1}=\dfrac{i^{2r}}{2^{2r}}=\dfrac{(-1)^r}{2^{2r}}=\dfrac{\mathrm{sen}\left(\dfrac{\pi n}{2}\right)}{2^{n-1}}.
        \]
    \end{proof}

\begin{teo}
Para $n\in \mathbb{N}$, temos

\[
\det(A(2,n))=
\begin{cases}
(-1)^{\frac{n}{2}} (n+1) &\,\,\, \textrm{se} \,\,\, n \,\,\,  \textrm{é par};\\
0 & \,\,\, \textrm{se} \,\,\, n \,\,\,  \textrm{é ímpar}.\
\end{cases}
\]
\end{teo}

\begin{proof}
Sabemos que

\begin{small}
\begin{align*}
\det(A(2,n))&=\displaystyle \prod_{k=1}^n\left[1+2\left(\cos\left(\frac{\pi}{3}\right)+\cos\left(\frac{k\pi}{n+1}\right)\right)\right]\left[1+2\left(\cos\left(\frac{2\pi}{3}\right)+\cos\left(\frac{k\pi}{n+1}\right)\right)\right]\\
&=\displaystyle \prod_{k=1}^n 4\cos\left(\frac{k\pi}{n+1}\right)\left[1+\cos\left(\frac{k\pi}{n+1}\right)\right].
\end{align*}
\end{small}
Vamos dividir em dois casos:

\begin{enumerate}
\item[Caso I.] Se $n$ é ímpar, então $n=2r+1$ para algum $r\in \mathbb{N}$. Como $1\leq k \leq 2r+1$, temos que algum $k$ é igual a $r+1.$ Assim,

\[
\cos\left(\frac{k\pi}{n+1}\right)=\cos\left(\frac{(r+1)\pi}{2(r+1)}\right)=\cos\left(\frac{\pi}{2}\right)=0.
\]
Portanto, $\det(A(2,n))=0$, se $n$ é ímpar.

\item[Caso II.] Se $n$ é par, então $n=2t$ para algum $t\in \mathbb{N}$. Então, usando a Proposição \ref{prop1} e algumas identidades trigonométricas obtemos

\begin{align*}
\det(A(2,2t))&=\displaystyle \prod_{k=1}^{2t} 4\cos\left(\frac{k\pi}{2t+1}\right)\left[1+\cos\left(\frac{k\pi}{2t+1}\right)\right]\\
&=4^{2t}\left[ \displaystyle \prod_{k=1}^{2t}\cos\left(\frac{k\pi}{2t+1}\right)\right]\left[\displaystyle \prod_{k=1}^{2t} \left(1+\cos\left(\frac{k\pi}{2t+1}\right)\right)\right]\\
&=4^{2t} \dfrac{\mathrm{sen}\left(\frac{(2t+1)\pi}{2}\right)}{2^{2t}}\left[\displaystyle \prod_{k=1}^{2t} \left(1+\cos\left(\frac{k\pi}{2t+1}\right)\right)\right]\\
&=2^{2t}(-1)^t \prod_{k=1}^{2t} \left(1+\cos\left(\frac{k\pi}{2t+1}\right)\right)\\
&=2^{2t}(-1)^t \prod_{k=1}^{2t} 2\cos^2\left(\dfrac{k\pi}{2(2t+1)}\right)\\
&=2^{2t} 2^{2t} (-1)^t \prod_{k=1}^{2t} \cos\left(\dfrac{k\pi}{2(2t+1)}\right) \mathrm{sen}\left(\dfrac{\pi}{2}-\dfrac{k\pi}{2(2t+1)}\right)\\
&\overset{s=2t+1-k}{=}2^{2t} 2^{2t} (-1)^t \prod_{s=1}^{2t} \cos\left(\dfrac{s\pi}{2(2t+1)}\right) \mathrm{sen}\left(\dfrac{s\pi}{2(2t+1)}\right)\\
&=2^{2t} 2^{2t} (-1)^t \prod_{s=1}^{2t} \dfrac{1}{2} \mathrm{sen}\left(\dfrac{s\pi}{2t+1}\right)\\
&=2^{2t} (-1)^t \prod_{s=1}^{2t}  \mathrm{sen}\left(\dfrac{s\pi}{2t+1}\right)=2^{2t} (-1)^t \dfrac{2t+1}{2^{2t}}=(-1)^t (2t+1).
\end{align*}
\end{enumerate}

\end{proof}

\section{Considerações finais}

O presente trabalho fornece um critério de solubilidade para o jogo 
\textit{Lights Out} no caso geral de uma malha 
$m$ por $n$, 
que depende da inversibilidade 
de uma certa matriz. Para tanto, fizemos uso de algumas ferramentas da álgebra linear. 
Adicionalmente, 
para um caso particular, calculamos o valor do determinante dessa matriz por meio de identidades 
trigonométricas.

A solubilidade do jogo para o caso de uma malha
$n$ por $n$
também foi abordada, 
por outros autores, usando a teoria da álgebra linear. Já para o caso geral, isto é, o jogo com uma 
malha
$m$ por $n$,
a solução foi desenvolvida usando outras teorias.

Acreditamos que o uso da Álgebra 
Linear propicia uma 
forma mais compreensível da solução do jogo. Por fim, o resultado deste trabalho mostra 
mais uma vez as implicações práticas de alguns conceitos da matemática.

Um próximo passo nesta direção seria investigar o valor do determinante das matrizes 
associadas ao jogo de modo mais geral e unificar a condição apresentada aqui com as condições dos outros autores.



\begin{thebibliography}{99}  

\bibitem{Anderson} \hypertarget{anderson}{ANDERSON}, Marlow; FEIL, Todd. Turning Lights Out with Linear Algebra. {\bf Mathematics Magazine}, [s.l.], v. 71, n. 4, p. 300-303, 1998. DOI: \url{https://doi.org/10.2307/2690705}.

\bibitem{Barua} \hypertarget{barua}{BARUA}, Rana; RAMAKRISHNAN, Subramanian. $\sigma$-game, $\sigma^+$-game and two-dimensional additive cellular automata. {\bf Theoretical Computer Science}, [s.l.], v. 154, n. 2, p. 349-366, 1996. DOI: \url{https://doi.org/10.1016/0304-3975(95)00091-7}.

\bibitem{Berger} \hypertarget{berger}{BERGER}, Arno. On linear independence of trigonometric numbers. {\bf Carpathian Journal of Mathematics}, [s.l.], v. 34, n. 2, p. 157-166, 2018. \label{Berger}

\bibitem{Berman} \hypertarget{berman}{BERMAN}, Abraham; BORER, Franziska; HUNGERBÜHLER, Norbert. Lights Out on graphs. {\bf Mathematische Semesterberichte}, [s.l.], v. 69, n. 2, p. 237-255, 2021. DOI: \url{https://doi.org/10.1007/s00591-021-00297-5}.

\bibitem{Fleischer} \hypertarget{fleischer}{FLEISCHER}, Rudolf; YU, Jiajin. A survey of the game ``Lights Out!''. {\it In:} BRODNIK, Andrej; LÓPEZ-ORTIZ, Alejandro; RAMAN, Venkatesh; VIOLA, Alfredo. {\bf Space-efficient data structures, streams, and algorithms}: Papers in honor of J. Ian Munro on the occasion of his 66th birthday. Berlim, Heidelberg: Springer Berlin Heidelberg, 2013. p. 176-198. DOI: \url{https://doi.org/10.1007/978-3-642-40273-9_13}

\bibitem{Goshima} \hypertarget{goshima}{GOSHIMA}, Masato; YAMAGISHI, Masakazu. Two remarks on torus Lights Out puzzle. {\bf Advances and Applications in Discrete Mathematics}, [s.l.], v. 71, n. 4, p. 115-126, 2009.

\bibitem{Kreh} \hypertarget{kreh}{KREH}, Martin. ``Lights Out'' and Variants. {\bf The American Mathematical Monthly}, [s.l.], v.~124, n.~10, p. 937-950, 2017.  \label{Kreh} DOI: \url{https://doi.org/10.4169/amer.math.monthly.124.10.937}.


\bibitem{Horn} \hypertarget{horn}{HORN}, Roger A.; JOHNSON, Charles R. {\bf Topics in Matrix Analysis}. Cambridge: Cambridge University Press, 1991. DOI: \url{https://doi.org/10.1017/CBO9780511840371}.

\bibitem{LEON} \hypertarget{leon}{LEON}, Steven J. {\bf álgebra linear com aplicações}. 8. ed. Rio de Janeiro: LTC, 2014.  

\bibitem{Madsen} \hypertarget{madsen}{MADSEN}, Matthew A. Lights Out: solutions using linear algebra. {\bf Summation}, Ripon, p. 36-40, 2010.

\bibitem{Noschese} \hypertarget{noschese}{NOSCHESE}, Silvia; PASQUINI, Lionello; REICHEL, Lothar. Tridiagonal Toeplitz matrices: properties and novel applications. {\bf Numerical Linear Algebra with Applications}, [s.l.], v. 20, n. 2, p. 302-326, 2013. DOI: \url{ https://doi.org/10.1002/nla.1811}.

\bibitem{sutner} \hypertarget{sutner}{SUTNER}, Klaus. Linear cellular automata and the Garden-of-Eden. {\bf The Mathematical Intelligencer}, [S.l.], v. 11, n. 2, p. 49-53, 1989. DOI: \url{https://doi.org/10.1007/BF03023823}. 





\end{thebibliography}

\end{document}